\documentclass{conm-p-l}

\usepackage{hyperref}

\newtheorem{theorem}{Theorem}[section]
\newtheorem{lemma}[theorem]{Lemma}

\theoremstyle{definition}

\theoremstyle{remark}
\newtheorem{remark}[theorem]{Remark}

\numberwithin{equation}{section}

\begin{document}

\title
[Analysis of first order systems]
{Analysis of first order systems of partial differential equations}


\author{Yan-Long Fang}
\address{Department of Mathematics,
University College London,
Gower Street,
London WC1E 6BT,
UK}
\email{yan.fang.12@ucl.ac.uk}

\author{Dmitri Vassiliev}
\address{Department of Mathematics,
University College London,
Gower Street,
London WC1E 6BT,
UK}
\email{D.Vassiliev@ucl.ac.uk}
\thanks{\url{http://www.homepages.ucl.ac.uk/\~ucahdva/};
Dmitri Vassiliev was supported by EPSRC grant EP/M000079/1}


\date{}

\begin{abstract}
The paper deals with a formally self-adjoint first order linear differential
operator acting on $m$-columns of complex-valued half-densities
over an $n$-manifold without boundary.
We study the distribution of eigenvalues in the elliptic setting
and the propagator in the hyperbolic setting,
deriving two-term asymptotic formulae for both.
We then turn our attention to the special case of a two by two operator in
dimension four. We show that
the geometric concepts of Lorentzian metric,
Pauli matrices,
spinor field,
connection coefficients for spinor fields,
electromagnetic covector potential, Dirac equation and Dirac action
arise naturally in the process of our analysis.
\end{abstract}

\maketitle

\section{The playing field}
\label{The playing field}

Let $L$ be a formally self-adjoint first order linear differential
operator acting on $m$-columns
$v=\begin{pmatrix}v_1&\ldots&v_m\end{pmatrix}^T$
of complex-valued half-densities
over a connected $n$-dimensional manifold $M$ without boundary.
Throughout this paper we assume that $m,n\ge2\,$.
The coefficients of the operator $L$ are assumed to be infinitely smooth.

In local coordinates $x=(x^1,\ldots,x^n)$ our operator reads
\begin{equation}
\label{operator in local coordinates}
L=P^\alpha(x)\frac\partial{\partial x^\alpha}+Q(x),
\end{equation}
where $P^\alpha(x)$ and $Q(x)$ are some $m\times m$ matrix-functions
and summation is carried out over $\alpha=1,\ldots,n$.
The full symbol of the operator $L$ is the matrix-function
\begin{equation}
\label{definition of the full symbol}
L(x,p):=iP^\alpha(x)\,p_\alpha+Q(x),
\end{equation}
where $p=(p_1,\ldots,p_n)$ is the dual variable (momentum).
Note that
the tradition in microlocal analysis is to denote momentum by $\xi$.
We choose to denote it by $p$ instead because
in Sections
\ref{Gauge-theoretic analysis in 3D}
and
\ref{A non-geometric representation of the massless Dirac action in 3D}
we will
need the letter $\xi$ for the spinor.

The problem with the full symbol is that it is not invariant under changes
of local coordinates. In order to overcome this problem we decompose
the full symbol into components homogeneous in momentum,
$L(x,p)=L_1(x,p)+L_0(x)$,
\begin{equation}
\label{definition of the homogeneous components of the full symbol}
L_1(x,p):=iP^\alpha(x)\,p_\alpha,
\qquad
L_0(x):=Q(x),
\end{equation}
and define the principal and subprincipal symbols as
\begin{equation}
\label{definition of the principal symbol}
L_\mathrm{prin}(x,p):=L_1(x,p),
\end{equation}
\begin{equation}
\label{definition of the subprincipal symbol}
L_\mathrm{sub}(x):=L_0(x)
+\frac i2(L_\mathrm{prin})_{x^\alpha p_\alpha}(x),
\end{equation}
where the subscripts indicate partial derivatives.
It is known that $L_\mathrm{prin}$ and $L_\mathrm{sub}$
are invariantly defined matrix-functions on $T^*M$ and $M$ respectively,
see subsection 2.1.3
in \cite{mybook} for details.
As we assumed our operator $L$ to be formally self-adjoint,
the matrix-functions $L_\mathrm{prin}$ and $L_\mathrm{sub}$
are Hermitian.

Examination of formulae
(\ref{operator in local coordinates})--(\ref{definition of the subprincipal symbol})
shows that
$L_\mathrm{prin}(x,p)$ and $L_\mathrm{sub}(x)$
uniquely determine the first order differential operator $L$.
Thus, the notions of principal symbol and subprincipal
symbol provide an invariant analytic way of describing a first order
differential operator.

We say that a formally self-adjoint first order differential operator $L$ is
\emph{elliptic}~if
\begin{equation}
\label{definition of ellipticity}
\det L_\mathrm{prin}(x,p)\ne0,\qquad\forall(x,p)\in T^*M\setminus\{0\},
\end{equation}
and \emph{non-degenerate} if
\begin{equation}
\label{definition of non-degeneracy}
L_\mathrm{prin}(x,p)\ne0,\qquad\forall(x,p)\in T^*M\setminus\{0\}.
\end{equation}
The ellipticity condition (\ref{definition of ellipticity})
is a standard condition in the spectral theory of differential operators,
see, for example, \cite{jst_part_a}.
Our non-degeneracy condition
(\ref{definition of non-degeneracy})
is less restrictive, and we will see later, in Section \ref{Dimension four is special},
that in the special case $m=2$ and $n=4$
this condition describes a hyperbolic operator.

\section{Distribution of eigenvalues and the propagator}
\label{Distribution of eigenvalues and the propagator}

In this section we assume that the manifold $M$ is compact (and without boundary)
and that the operator $L$ is elliptic, see (\ref{definition of ellipticity}).

\begin{remark}
\label{remark about m being even}
Ellipticity and the fact that dimension $n$ is greater than or equal to two imply that $m$,
the number of equations, is even. Indeed, let us fix an arbitrary point $x\in M$
and consider $L_\mathrm{prin}(x,p)$ as a function of momentum
$p\in T_x^*M$.
As throughout this paper
the operator $L$ is assumed to be formally self-adjoint,
the matrix-function $L_\mathrm{prin}(x,p)$ is Hermitian, 
and, hence, $\det L_\mathrm{prin}(x,p)$ is real.
For $n\ge2$ the set
$T_x^*M\setminus\{0\}$ is connected,
so the ellipticity condition \eqref{definition of ellipticity}
implies that the polynomial $\det L_\mathrm{prin}(x,p)$
preserves sign on $T_x^*M\setminus\{0\}$.
But our $m\times m$ matrix-function $L_\mathrm{prin}(x,p)$ is linear in $p$, so
$\det L_\mathrm{prin}(x,-p)=(-1)^m\det L_\mathrm{prin}(x,p)$,
therefore the sign of $\det L_\mathrm{prin}(x,p)$ can only be preserved if $m$
is even.
\end{remark}

Let $h^{(j)}(x,p)$ be the eigenvalues of the principal symbol. We
assume that these are
simple for all $(x,p)\in T^*M\setminus\{0\}$.
We enumerate the eigenvalues of the principal symbol
$h^{(j)}(x,p)$ in increasing order, using
a negative index $j=-m/2,\ldots,-1$ for negative $h^{(j)}(x,p)$
and
a positive index $j=1,\ldots,m/2$ for positive $h^{(j)}(x,p)$.

It is known that our differential operator $L$ has a discreet
spectrum accumulating to $+\infty$ and to $-\infty$.
Let $\lambda_k$ and $v_k=\begin{pmatrix}v_{k1}(x)&\ldots&v_{km}(x)\end{pmatrix}^T$ be
the eigenvalues and eigenfunctions of the operator $L$; the
particular enumeration of these eigenvalues
(accounting for multiplicities)
is irrelevant for our purposes.

We will be studying the following two objects.

\

\textbf{Object 1.}
Our first object of study is the \emph{propagator},
which is the one-parameter family of operators defined as
\begin{multline}
\label{definition of wave group}
U(x^{n+1}):=e^{-ix^{n+1}L}
\\
=\sum_k e^{-ix^{n+1}\lambda_k}\,v_k(x^1,\ldots,x^n)\int_M[v_k(y^1,\ldots,y^n)]^*(\,\cdot\,)\,dy^1\ldots dy^n\,,
\end{multline}
where $x^{n+1}\in\mathbb{R}$ is an additional `time' coordinate.
The propagator provides a solution to the Cauchy problem
\begin{equation}
\label{initial condition most basic}
\left.w\right|_{x^{n+1}=0}=v
\end{equation}
for the hyperbolic system
\begin{equation}
\label{dynamic equation most basic}
(-i\partial/\partial x^{n+1}+L)w=0\,.
\end{equation}
Namely, it is easy to see that if the column of half-densities $v=v(x^1,\ldots,x^n)$
is infinitely smooth,
then, setting
$\,w:=U(x^{n+1})\,v$, we get a `time-dependent' column of half-densities $w(x^1,\ldots,x^n,x^{n+1})$
which is also infinitely smooth
and which satisfies the equation
(\ref{dynamic equation most basic})
and the initial condition
(\ref{initial condition most basic}).
The use of the letter ``$U$'' for the propagator is motivated by the
fact that for each $x^{n+1}$ the operator $U(x^{n+1})$ is unitary.

Note that the operator $-i\partial/\partial x^{n+1}+L$ appearing
in the LHS of formula (\ref{dynamic equation most basic})
is a formally self-adjoint $m\times m$ first order differential
operator on the $(n+1)$-dimensional manifold $M\times\mathbb{R}$.
Moreover, it is easy to see that
this `extended' operator $-i\partial/\partial x^{n+1}+L$ automatically satisfies the
non-degeneracy condition from Section~\ref{The playing field}.

\

\textbf{Object 2.}
Our second object of study is the \emph{counting function}
\begin{equation}
\label{definition of counting function}
N(\lambda):=\,\sum_{0<\lambda_k<\lambda}1\,.
\end{equation}
In other words, $N(\lambda)$ is the number of eigenvalues $\lambda_k$
between zero and a positive~$\lambda$.

Here it is natural to ask the question: why, in defining the counting function
(\ref{definition of counting function}),
did we choose to count
all \emph{positive} eigenvalues up to a given positive $\lambda$
rather than
all \emph{negative} eigenvalues up to a given negative $\lambda$?
There is no particular reason. One case reduces to the other by the change
of operator $L\mapsto-L$. This issue is known as
\emph{spectral asymmetry} and is discussed in
\cite{atiyah_part_3},
as well as in Section~10 of \cite{jst_part_a}
and in \cite{torus}.

\

Our objectives are as follows.

\

\textbf{Objective 1.}
We aim to construct the propagator
(\ref{definition of wave group})
explicitly in terms of
oscillatory integrals, modulo an integral operator with an
infinitely smooth, in the variables $x^1,\ldots,x^n,x^{n+1},y^1,\ldots,y^n$, integral kernel.

\

\textbf{Objective 2.}
We aim to derive a
two-term asymptotic expansion for the counting function
(\ref{definition of counting function})
\begin{equation}
\label{two-term asymptotic formula for counting function}
N(\lambda)=a\lambda^n+b\lambda^{n-1}+o(\lambda^{n-1})
\end{equation}
as $\lambda\to+\infty$, where $a$ and $b$ are some real constants.
More specifically, our objective is to write down explicit formulae
for the asymptotic coefficients $a$ and $b$.

Here one has to have in mind that the two-term asymptotic
expansion (\ref{two-term asymptotic formula for counting function})
holds only under appropriate assumptions on periodic trajectories,
see Theorem~8.4 from \cite{jst_part_a} for details.
In order to avoid dealing with the issue of periodic trajectories,
in this paper we understand the asymptotic expansion
(\ref{two-term asymptotic formula for counting function})
in a regularised fashion.
One way of regularising
the asymptotic formula (\ref{two-term asymptotic formula for counting function})
is to take a convolution with a function from Schwartz space
$\mathcal{S}(\mathbb{R})$; see Theorem~7.2 in \cite{jst_part_a} for details.
Alternatively, one can look at the eta function
$\,\eta(s):=\sum|\lambda_k|^{-s}\operatorname{sign}\lambda_k\,$,
where summation is carried out over all nonzero eigenvalues
$\lambda_k$ and $s\in\mathbb{C}$ is the independent variable.
The series converges absolutely for
$\operatorname{Re}s>n$ and defines a holomorphic function in this half-plane.
Moreover, it is known
\cite{atiyah_part_3}
that the eta function extends meromorphically to the whole $s$-plane
with simple poles.
Formula (10.6) from \cite{jst_part_a} implies
that the eta function does not have a pole at $s=n$ and
that the residue at 
$s=n-1$ is $2(n-1)b$,
where $b$ is the coefficient from
(\ref{two-term asymptotic formula for counting function}).

\

It is well known that the above two objectives are closely
related: if one achieves Objective 1, then Objective 2 follows via a
Fourier transform in the variable $x^{n+1}$, see Sections 6 and 7
in \cite{jst_part_a}.

\

We are now in a position to state our results.

\

\textbf{Result 1.}
We construct the propagator as a sum of $m$ oscillatory integrals
(Fourier integral operators)
\begin{equation*}
\label{wave group as a sum of oscillatory integrals}
U(x^{n+1})\overset{\operatorname{mod}C^\infty}=
\sum_j
U^{(j)}(x^{n+1})\,,
\end{equation*}
where the phase function of each oscillatory integral
$U^{(j)}(x^{n+1})$ is associated with the corresponding
Hamiltonian $h^{(j)}(x^1,\ldots,x^n,p_1,\ldots,p_n)$
and summation is
performed
over
nonzero integers $j$
from $-m/2$ to $+m/2$.
The notion of a phase function associated with a Hamiltonian is
defined in Section 2 of \cite{jst_part_a}
and Section~2.4 of \cite{mybook}.

We will now write down explicitly the principal symbol of the oscillatory
integral $U^{(j)}(x^{n+1})$.
The notion of a principal symbol of an oscillatory
integral is defined in accordance with Definition 2.7.12 from
\cite{mybook}.
The principal symbol of the oscillatory
integral $U^{(j)}(x^{n+1})$ is a complex-valued
$m\times m$ matrix-function on
\[
M\times\mathbb{R}\times(T^*M\setminus\{0\}).
\]
We denote the arguments of this principal symbol by
$x^1,\ldots,x^n$ (local coordinates on~$M$),
$x^{n+1}$ (`time' coordinate on $\mathbb{R}$),
$y^1,\ldots,y^n$ (local coordinates on $M$)
and
$q_1,\ldots,q_n$ (variable dual to $y^1,\ldots,y^n$).

Further on in this section we use $x$, $y$, $p$ and $q$ as shorthand for
$\,x^1,\ldots,x^n$,
$\,y^1,\ldots,y^n$,
$\,p_1,\ldots,p_n$
and
$\,q_1,\ldots,q_n$
respectively.
The additional `time' coordinate $x^{n+1}$ will always be written separately.

In order to write down the principal symbol of the oscillatory
integral $U^{(j)}(x^{n+1})$ we need to introduce some auxiliary objects first.

Curly brackets will denote the Poisson bracket on matrix-functions
$\{P,R\}:=P_{x^\alpha}R_{p_\alpha}-P_{p_\alpha}R_{x^\alpha}$
and its further generalisation
\begin{equation}
\label{definition of generalised Poisson bracket}
\{F,G,H\}:=F_{x^\alpha}GH_{p_\alpha}-F_{p_\alpha}GH_{x^\alpha},
\end{equation}
where the subscripts $x^\alpha$ and $p_\alpha$
indicate partial derivatives and
the repeated index $\alpha$ indicates summation over $\alpha=1,\ldots,n$.

Let $v^{(j)}(x,p)$ be the normalised eigenvector of the principal
symbol $L_\mathrm{prin}(x,p)$ corresponding to the eigenvalue
$h^{(j)}(x,p)$. We define the scalar function
$f^{(j)}:T^*M\setminus\{0\}\to\mathbb{R}$ in accordance with the formula
\begin{equation*}
\label{phase appearing in principal symbol}
f^{(j)}:=[v^{(j)}]^*L_\mathrm{sub}v^{(j)}
-\frac i2
\{
[v^{(j)}]^*,L_\mathrm{prin}-h^{(j)},v^{(j)}
\}
-i[v^{(j)}]^*\{v^{(j)},h^{(j)}\}.
\end{equation*}

By
$(x^{(j)}(x^{n+1};y,q),p^{(j)}(x^{n+1};y,q))$ we denote the Hamiltonian trajectory
originating from the point $(y,q)$, i.e.~solution of the system of
ordinary differential equations (the dot denotes differentiation in $x^{n+1}$)
\begin{equation*}
\label{Hamiltonian system of equations}
\dot x^{(j)}=h^{(j)}_p(x^{(j)},p^{(j)}),
\qquad
\dot p^{(j)}=-h^{(j)}_x(x^{(j)},p^{(j)})
\end{equation*}
subject to the initial condition $\left.(x^{(j)},p^{(j)})\right|_{x^{n+1}=0}=(y,q)$.

The formula for the principal symbol of the oscillatory integral
$U^{(j)}(x^{n+1})$ is known
\cite{SafarovDSc,NicollPhD,jst_part_a}
and reads as follows:
\begin{multline}
\label{formula for principal symbol of oscillatory integral}
[v^{(j)}(x^{(j)}(x^{n+1};y,q),p^{(j)}(x^{n+1};y,q))]
\,[v^{(j)}(y,q)]^*
\\
\times\exp
\left(
-i\int_0^{x^{n+1}}f^{(j)}(x^{(j)}(\tau;y,q),p^{(j)}(\tau;y,q))\,d\tau
\right).
\end{multline}
This principal symbol is positively homogeneous in momentum $q$
of degree zero.

Let us now examine the lower order terms of the symbol of the
oscillatory integral $U^{(j)}(x^{n+1})$. The algorithm described
in Section 2 of \cite{jst_part_a} provides a recursive procedure
for the calculation of all lower order terms, of any degree of homogeneity
in momentum $q$. However, there are two issues here. Firstly, calculations
become very complicated.
Secondly, describing these lower
order terms in an invariant way is problematic because, as far as
the authors are aware, the concept of subprincipal symbol has
never been defined for time-dependent oscillatory integrals
(Fourier integral operators). We thank Yuri Safarov for
drawing our attention to the latter issue.

We overcome the problem of invariant description of lower order
terms of the symbol of the oscillatory integral $U^{(j)}(x^{n+1})$
by restricting our analysis to $U^{(j)}(0)$. It turns out that
knowing the properties of the lower order terms of the symbol of
$U^{(j)}(0)$ is sufficient for the derivation of the two-term asymptotic
expansion (\ref{two-term asymptotic formula for counting function}).
And $U^{(j)}(0)$ is a pseudodifferential operator, so one can use
here the standard notion of subprincipal symbol of a
pseudodifferential operator,
see subsection 2.1.3 in \cite{mybook} for definition.

The following result was established recently in \cite{jst_part_a}.

\begin{theorem}
\label{theorem2p1}
\begin{equation}
\label{subprincipal symbol of OI at time zero}
\operatorname{tr}[U^{(j)}(0)]_\mathrm{sub}
=-i\{[v^{(j)}]^*,v^{(j)}\}.
\end{equation}
\end{theorem}

It is interesting that the RHS of formula
(\ref{subprincipal symbol of OI at time zero})
admits a geometric interpretation: it can be interpreted
as the scalar curvature of a $\mathrm{U}(1)$ connection
on $T^*M\setminus\{0\}$, see Section 5 of \cite{jst_part_a}
for details. This connection is to do with gauge transformations
of the normalised eigenvector
$v^{(j)}(x,p)$ of the principal
symbol $L_\mathrm{prin}(x,p)$ corresponding to the eigenvalue
$h^{(j)}(x,p)$. Namely, observe that if $v^{(j)}(x,p)$ is an eigenvector
and $\phi^{(j)}(x,p)$ is an arbitrary real-valued function,
then $e^{i\phi^{(j)}(x,p)}v^{(j)}(x,p)$ is also an eigenvector,
and careful analysis of this gauge transformation leads to the appearance of
a curvature term.

\

\textbf{Result 2.}
The formula for the first coefficient of the asymptotic expansion
(\ref{two-term asymptotic formula for counting function})
reads
\cite{ivrii_springer_lecture_notes,ivrii_book}
\begin{equation}
\label{formula for a}
a=(2\pi)^{-n}\sum_{j=1}^{m/2}
\ \int\limits_{h^{(j)}(x,p)<1}dx\,dp\,,
\end{equation}
where $dx=dx^1\ldots dx^n$ and $dp=dp_1\ldots dp_n\,$.

The formula for the second coefficient of the asymptotic expansion
(\ref{two-term asymptotic formula for counting function})
was established recently in \cite{jst_part_a}.

\begin{theorem}
\label{theorem2p2}
\begin{multline}
\label{formula for b}
b=-n(2\pi)^{-n}\sum_{j=1}^{m/2}
\ \int\limits_{h^{(j)}(x,p)<1}
\Bigl(
[v^{(j)}]^*L_\mathrm{sub}v^{(j)}
\\
-\frac i2
\{
[v^{(j)}]^*,L_\mathrm{prin}-h^{(j)},v^{(j)}
\}
+\frac i{n-1}h^{(j)}\{[v^{(j)}]^*,v^{(j)}\}
\Bigr)(x,p)\,
dx\,dp\,.
\end{multline}
\end{theorem}

Note that Theorem \ref{theorem2p1} plays an important role
in the proof of Theorem \ref{theorem2p2}.  The information contained
in formula (\ref{formula for principal symbol of oscillatory integral})
is, on its own,
insufficient for the derivation of the formula for the coefficient $b$.

Recall also that according to Remark \ref{remark about m being even}
the number $m$ is even, so the upper limit of
summation in formulae
\eqref{formula for a}
and
\eqref{formula for b}
is a natural number.

A bibliographic review of the subject is provided in Section 11
of  \cite{jst_part_a}.

\section{Two by two operators are special}
\label{Two by two operators are special}

Suppose that we are dealing with a $2\times2$ operator, i.e.~suppose that
\begin{equation}
\label{m equals two}
m=2.
\end{equation}
Observe that in this case the determinant of the principal symbol is a quadratic form
in the dual variable (momentum) $p$\,:
\begin{equation}
\label{definition of metric}
\det L_\mathrm{prin}(x,p)=-g^{\alpha\beta}(x)\,p_\alpha p_\beta\,.
\end{equation}
We interpret the real coefficients $g^{\alpha\beta}(x)=g^{\beta\alpha}(x)$,
$\alpha,\beta=1,\ldots,n$, appearing in formula (\ref{definition of metric})
as components of a (contravariant) metric tensor.
Thus, $2\times2$ formally self-adjoint first order linear partial differential
operators are special in that the concept of a metric is encoded
within such operators. This opens the way to the geometric
interpretation of analytic results.
So further on in this paper we work under the assumption (\ref{m equals two}).

\section{Dimension four is special}
\label{Dimension four is special}

It is easy to see that if $n\ge5$, then our metric defined in
accordance with formula (\ref{definition of metric}) has the
property
$\det g^{\alpha\beta}(x)=0$, $\forall x\in M$.
Hence,
\begin{equation}
\label{n equals four}
n=4
\end{equation}
is the highest dimension in which it makes sense to define the
metric as we do.
So further on in this section and the next two we work under the assumption (\ref{n equals four}).

It is natural to ask the question: what is the signature of our
metric? The answer is given by the following lemma.

\begin{lemma}
\label{Lemma about Lorentzian metric}
Suppose that we have
(\ref{m equals two}) and (\ref{n equals four})
and suppose that our operator $L$ satisfies the non-degeneracy
condition (\ref{definition of non-degeneracy}).
Then our metric tensor defined in
accordance with formula (\ref{definition of metric}) is Lorentzian, i.e.~it has
three positive eigenvalues
and one negative eigenvalue.
\end{lemma}

Lemma~\ref{Lemma about Lorentzian metric} is proved
by a straightforward calculation, see Section 2 in \cite{nongeometric}.

Lemma~\ref{Lemma about Lorentzian metric}
tells us that under the conditions
(\ref{m equals two}),
 (\ref{n equals four})
and (\ref{definition of non-degeneracy})
our operator $L$ is hyperbolic.
This indicates that one could, in principle,
perform a comprehensive microlocal analysis
of the corresponding propagator and, moreover, do this
in a relativistically invariant fashion.
Here the relativistic propagator is, loosely speaking,
the Fourier integral operator mapping
a 2-column of half-densities $v$
to a 2-column of half-densities $w$
which is a solution of the hyperbolic system $Lw=v$.
One would expect the construction of this relativistic propagator
to proceed along the lines of the construction sketched out
in Section \ref{Distribution of eigenvalues and the propagator}
and, in more detail, in \cite{jst_part_a},
only without reference to a particular choice of time coordinate.

We are currently a long way from developing
relativistic microlocal techniques. However, we are able to
perform a gauge-theoretic analysis of
$2\times2$ operators in dimension four.
The results of this gauge-theoretic analysis
are presented in the next two sections
and these results reveal additional geometric structures
encoded within the operator $L$. We hope that the
identification of these geometric structures
will eventually help us develop relativistic
microlocal techniques.

\section{Gauge-theoretic analysis in 4D}
\label{Gauge-theoretic analysis in 4D}

Take an arbitrary matrix-function
\begin{equation}
\label{GL2C matrix-function Q}
Q:M\to\mathrm{GL}(2,\mathbb{C})
\end{equation}
and consider the transformation of our $2\times2$ differential operator
\begin{equation}
\label{GL2C transformation of the operator}
L\mapsto Q^*LQ.
\end{equation}

We interpret
(\ref{GL2C transformation of the operator}) as a gauge transformation.
Note that in spectral theory it is customary to apply unitary transformations
rather than general linear transformations.
However, in view of
Lemma~\ref{Lemma about Lorentzian metric}
we are working in a relativistic (hyperbolic) setting without
a specified time coordinate and in this setting
restricting our analysis to unitary transformations would be unnatural.

The transformation (\ref{GL2C transformation of the operator})
of the differential operator $L$ induces the following transformations
of its principal and subprincipal symbols:
\begin{equation}
\label{GL2C transformation of the principal symbol}
L_\mathrm{prin}\mapsto Q^*L_\mathrm{prin}Q,
\end{equation}
\begin{equation}
\label{GL2C transformation of the subprincipal symbol}
L_\mathrm{sub}\mapsto
Q^*L_\mathrm{sub}Q
+\frac i2
\left(
Q^*_{x^\alpha}(L_\mathrm{prin})_{p_\alpha}Q
-
Q^*(L_\mathrm{prin})_{p_\alpha}Q_{x^\alpha}
\right).
\end{equation}

Comparing formulae
(\ref{GL2C transformation of the principal symbol})
and
(\ref{GL2C transformation of the subprincipal symbol})
we see that, unlike the principal symbol, the subprincipal
symbol does not transform in a covariant fashion due to
the appearance of terms with the gradient of the
matrix-function $Q(x)$. In order to identify the sources of this
non-covariance we observe that
any matrix-function (\ref{GL2C matrix-function Q})
can be written as a product of three terms:
a complex matrix-function of determinant one,
a positive scalar function and
a complex scalar function of modulus one.
Hence, we examine the three gauge-theoretic actions separately.

Take an arbitrary scalar function
$\psi:M\to\mathbb{R}$
and consider the transformation of our differential operator
\begin{equation}
\label{psi transformation of the operator}
L\mapsto e^\psi Le^\psi.
\end{equation}
The transformation
(\ref{psi transformation of the operator}) is a special case of the transformation
(\ref{GL2C transformation of the operator}) with $Q=e^\psi I$,
where $I$ is the $2\times2$ identity matrix.
Substituting this $Q$ into formula
(\ref{GL2C transformation of the subprincipal symbol}),
we get
\begin{equation}
\label{psi transformation of the subprincipal symbol}
L_\mathrm{sub}\mapsto
e^{2\psi}L_\mathrm{sub},
\end{equation}
so the subprincipal symbol transforms in a covariant fashion.

Now take an arbitrary scalar function
$\phi:M\to\mathbb{R}$
and consider the transformation of our differential operator
\begin{equation}
\label{phi transformation of the operator}
L\mapsto e^{-i\phi}Le^{i\phi}.
\end{equation}
The transformation
(\ref{phi transformation of the operator}) is a special case of the transformation
(\ref{GL2C transformation of the operator}) with $Q=e^{i\phi}I$.
Substituting this $Q$ into formula
(\ref{GL2C transformation of the subprincipal symbol}),
we get
\begin{equation}
\label{phi transformation of the subprincipal symbol}
L_\mathrm{sub}(x)\mapsto
L_\mathrm{sub}(x)+L_\mathrm{prin}(x,(\operatorname{grad}\phi)(x)),
\end{equation}
so the subprincipal symbol does not transform in a covariant fashion.
We do not (and can not) take any action with regards to the non-covariance
of (\ref{phi transformation of the subprincipal symbol}).

Finally, take an arbitrary matrix-function
$R:M\to\mathrm{SL}(2,\mathbb{C})$
and consider the transformation of our differential operator
\begin{equation}
\label{SL2C transformation of the operator}
L\mapsto R^*LR.
\end{equation}
Of course, the transformation
(\ref{SL2C transformation of the operator}) is a special case of the transformation
(\ref{GL2C transformation of the operator}): we are looking at the
case when $\det Q(x)=1$.
It turns out that it is possible to overcome the resulting non-covariance
in (\ref{GL2C transformation of the subprincipal symbol}) by introducing
the \emph{covariant subprincipal symbol} $\,L_\mathrm{csub}(x)\,$
in accordance with formula
\begin{equation}
\label{definition of covariant subprincipal symbol}
L_\mathrm{csub}:=
L_\mathrm{sub}
+\frac i{16}\,
g_{\alpha\beta}
\{
L_\mathrm{prin}
,
\operatorname{adj}L_\mathrm{prin}
,
L_\mathrm{prin}
\}_{p_\alpha p_\beta},
\end{equation}
where subscripts $p_\alpha$, $p_\beta$ indicate partial derivatives,
curly brackets denote the generalised Poisson bracket on matrix-functions
(\ref{definition of generalised Poisson bracket})
and $\,\operatorname{adj}\,$ stands for the operator of matrix adjugation
\begin{equation}
\label{definition of adjugation}
P=\begin{pmatrix}a&b\\ c&d\end{pmatrix}
\mapsto
\begin{pmatrix}d&-b\\-c&a\end{pmatrix}
=:\operatorname{adj}P
\end{equation}
from elementary linear algebra.

\begin{lemma}
\label{Lemma about covariant subprincipal symbol}
The transformation
(\ref{SL2C transformation of the operator})
of the differential operator induces the transformation
\begin{equation}
\label{SL2C transformation of the covariant subprincipal symbol}
L_\mathrm{csub}\mapsto
R^*L_\mathrm{csub}R
\end{equation}
of its covariant subprincipal symbol.
\end{lemma}

The proof of
Lemma~\ref{Lemma about covariant subprincipal symbol}
is given in \cite{nongeometric}.

Formula (\ref{SL2C transformation of the covariant subprincipal symbol})
tells us that when working
with $2\times2$ operators in dimension four
it makes sense to use the covariant subprincipal
symbol rather than the standard  subprincipal
symbol because the covariant subprincipal
is `more invariant'.

It is easy to check that,
like the standard  subprincipal symbol,
the covariant subprincipal symbol
of a formally self-adjoint
non-degenerate first order differential operator
is an Hermitian matrix-function on the base manifold $M$.
Also, it is easy to see that
$L_\mathrm{prin}(x,p)$ and $L_\mathrm{csub}(x)$
uniquely determine our operator $L$.

Substituting (\ref{definition of the subprincipal symbol})
into (\ref{definition of covariant subprincipal symbol})
we get
\begin{equation}
\label{definition of covariant subprincipal symbol extended}
L_\mathrm{csub}=
L_0
+\frac i2(L_\mathrm{prin})_{x^\alpha p_\alpha}
+\frac i{16}\,
g_{\alpha\beta}
\{
L_\mathrm{prin}
,
\operatorname{adj}L_\mathrm{prin}
,
L_\mathrm{prin}
\}_{p_\alpha p_\beta}.
\end{equation}
Comparing formulae
(\ref{definition of the subprincipal symbol})
and
(\ref{definition of covariant subprincipal symbol extended})
we see that the standard subprincipal symbol
and covariant subprincipal symbol have the same structure, only
the covariant subprincipal symbol has a second correction term
designed to `take care' of special linear transformations.

Examination of formulae
(\ref{definition of covariant subprincipal symbol}),
(\ref{psi transformation of the subprincipal symbol})
and
(\ref{phi transformation of the subprincipal symbol})
shows that the transformations
(\ref{psi transformation of the operator})
and
(\ref{phi transformation of the operator})
of the differential operator induce the transformations
\begin{equation}
\label{psi transformation of the covariant subprincipal symbol}
L_\mathrm{csub}\mapsto
e^{2\psi}L_\mathrm{csub}\,,
\end{equation}
\begin{equation}
\label{phi transformation of the covariant subprincipal symbol}
L_\mathrm{csub}(x)\mapsto
L_\mathrm{csub}(x)
+L_\mathrm{prin}(x,(\operatorname{grad}\phi)(x))
\end{equation}
of its covariant subprincipal symbol.
Thus, the switch from standard subprincipal symbol to covariant
subprincipal symbol `does not spoil' the behaviour under the
transformations
(\ref{psi transformation of the operator})
and
(\ref{phi transformation of the operator}).

The non-degeneracy condition (\ref{definition of non-degeneracy})
implies that for each $x\in M$ the matrices
\begin{equation}
\label{new formula for Pauli matrices}
\sigma^\alpha(x):=(L_\mathrm{prin})_{p_\alpha}(x),\qquad\alpha=1,2,3,4,
\end{equation}
form a basis
in the real vector space of $2\times2$ Hermitian matrices.
Decomposing the covariant subprincipal symbol $L_\mathrm{csub}(x)$
with respect to this basis, we get
$\,L_\mathrm{csub}(x)=\sigma^\alpha(x)\,A_\alpha(x)\,$
with some real coefficients $A_\alpha(x)$, $\alpha=1,2,3,4$.
The latter formula
can be rewritten in more compact form as
\begin{equation}
\label{decomposition of Z compact}
L_\mathrm{csub}(x)=L_\mathrm{prin}(x,A(x)),
\end{equation}
where $A$ is a covector field with components $A_\alpha(x)$, $\alpha=1,2,3,4$.
Formula (\ref{decomposition of Z compact})
tells us that the covariant subprincipal symbol $L_\mathrm{csub}$
is equivalent to a real-valued covector field $A$,
the electromagnetic covector potential.

Examination of formulae
(\ref{SL2C transformation of the covariant subprincipal symbol})
and
(\ref{psi transformation of the covariant subprincipal symbol})--(\ref{decomposition of Z compact})
shows that our electromagnetic covector potential $A$ is invariant
under the transformations
(\ref{SL2C transformation of the operator})
and
(\ref{psi transformation of the operator})
of our differential operator $L$,
whereas
the transformation (\ref{phi transformation of the operator})
induces the transformation
$A\mapsto A+\operatorname{grad}\phi$.

The geometric and theoretical physics interpretations of the transformations
(\ref{psi transformation of the operator}),
(\ref{phi transformation of the operator})
and
(\ref{SL2C transformation of the operator})
are discussed in detail in \cite{nongeometric}.

\section{A non-geometric representation of the massive Dirac equation in 4D}
\label{A non-geometric representation of the massive Dirac equation in 4D}

As explained in the previous section,
in dimension four our formally self-adjoint non-degenerate $2\times2$
first order differential operator $L$
is completely determined by its principal symbol $L_\mathrm{prin}(x,p)$
and covariant subprincipal symbol $L_\mathrm{csub}(x)$.
Namely, in local coordinates the formula for the differential
operator $L$ reads
\begin{multline}
\label{symbol Op explicit}
L=
-i
[(L_\mathrm{prin})_{p_\alpha}(x)]\frac\partial{\partial x^\alpha}
\\
-\frac i2(L_\mathrm{prin})_{x^\alpha p_\alpha}(x)
-\frac i{16}
(
g_{\alpha\beta}
\{
L_\mathrm{prin}
,
\operatorname{adj}L_\mathrm{prin}
,
L_\mathrm{prin}
\}_{p_\alpha p_\beta}
)
(x)
+L_\mathrm{csub}(x)\,.
\end{multline}
Further on we use the notation
\begin{equation}
\label{symbol Op}
L=\operatorname{Op}(L_\mathrm{prin},L_\mathrm{csub})
\end{equation}
as shorthand for (\ref{symbol Op explicit}).
We call (\ref{symbol Op}) the \emph{covariant representation}
of the differential operator $L$.

Using the covariant representation (\ref{symbol Op})
and matrix adjugation (\ref{definition of adjugation})
we define the
adjugate of the differential operator $L$ as
\begin{equation*}
\label{definition of adjugate operator}
\operatorname{Adj}L:=\operatorname{Op}
(\operatorname{adj}L_\mathrm{prin},
\operatorname{adj}L_\mathrm{csub}).
\end{equation*}

We define the Dirac operator as the differential operator
\begin{equation}
\label{analytic definition of the Dirac operator}
D:=
\begin{pmatrix}
L&mI\\
mI&\operatorname{Adj}L
\end{pmatrix}
\end{equation}
acting on 4-columns $v$ of complex-valued scalar fields.
Here $m$ is the electron mass and
$I$ is the $2\times2$ identity matrix.

We claim
that the system of four scalar equations
$Dv=0$
is equivalent to the Dirac equation in its traditional geometric formulation.
In order to justify this claim we need to compare our Dirac operator
(\ref{analytic definition of the Dirac operator}) with the traditional Dirac
operator $D_\mathrm{trad}$, see Appendix A in \cite{nongeometric}
for definition.

\begin{theorem}
\label{theorem6p1}
The two operators
are related by the formula
\begin{equation*}
\label{main theorem formula}
D
=
|\det g_{\kappa\lambda}|^{1/4}
\,
D_\mathrm{trad}
\,
|\det g_{\mu\nu}|^{-1/4}\,,
\end{equation*}
where the Lorentzian metric is defined in accordance with formula (\ref{definition of metric}).
\end{theorem}

The proof of
Theorem~\ref{theorem6p1}
is given in \cite{nongeometric}.

Our representation (\ref{analytic definition of the Dirac operator})
of the massive hyperbolic Dirac operator in dimension four
is given in an analytic language different from the traditional geometric language.
Hence, we
feel the need to reassure the reader that all the standard ingredients
are implicitly contained in (\ref{analytic definition of the Dirac operator}).

The matrices
(\ref{new formula for Pauli matrices})
are our Pauli matrices. Moreover, it is easy to see that
our definition of the metric
(\ref{definition of metric})
ensures that our Pauli matrices
(\ref{new formula for Pauli matrices})
automatically satisfy the standard
defining relation
$
\sigma^\alpha(\operatorname{adj}\sigma^\beta)
+
\sigma^\beta(\operatorname{adj}\sigma^\alpha)
=-2Ig^{\alpha\beta}
$.

The traditional representation of the Dirac operator
involves covariant derivatives of spinor fields with respect
to the Levi-Civita connection.
Technical calculations given in  \cite{nongeometric}
show that these connection coefficients are contained within
the Poisson bracket term in
our definition of the covariant subprincipal symbol
(\ref{definition of covariant subprincipal symbol}).
More precisely,
the Poisson bracket term in
formula (\ref{definition of covariant subprincipal symbol})
does not give each spinor connection coefficient separately,
it rather gives their sum, the way they appear in the Dirac operator.

\section{Dimension three is special}
\label{Dimension three is special}

In the remainder of the paper we retain the assumption (\ref{m equals two}),
and we also assume that
\begin{equation}
\label{n equals three}
n=3
\end{equation}
and that the principal symbol of
our formally self-adjoint
first order differential operator $L$
is trace-free,
\begin{equation}
\label{principal symbol is trace-free}
\operatorname{tr}L_\mathrm{prin}(x,p)=0.
\end{equation}

In addition, in the remainder of the paper
we assume ellipticity (\ref{definition of ellipticity}).
Note that under the assumptions
(\ref{m equals two}),
(\ref{n equals three})
and
(\ref{principal symbol is trace-free})
the ellipticity condition
(\ref{definition of ellipticity})
is equivalent to the non-degeneracy condition
(\ref{definition of non-degeneracy}).

We define the metric tensor in accordance with formula (\ref{definition of metric}).
It is easy to see that under the assumptions
(\ref{m equals two}),
(\ref{n equals three}),
(\ref{principal symbol is trace-free})
and
(\ref{definition of ellipticity})
our metric is Riemannian, i.e.~the metric tensor is positive definite.

\section{Gauge-theoretic analysis in 3D}
\label{Gauge-theoretic analysis in 3D}

The metric
tensor defined in accordance with formula (\ref{definition of metric})
does not determine the Hermitian matrix-function $L_\mathrm{prin}(x,p)$
uniquely.
Hence, in this section we identify a further geometric object
encoded within
the principal symbol of our differential operator $L$.
To this end,
we will now start varying this principal symbol,
assuming the metric $g$, defined by formula
(\ref{definition of metric}), to be fixed (prescribed).

Let us fix a reference principal symbol $\mathring L_\mathrm{prin}(x,p)$ corresponding to the
prescribed metric $g$ and look at all principal symbols $L_\mathrm{prin}(x,p)$ which
correspond to the same prescribed metric $g$.

\begin{lemma}
\label{Lemma about two principal symbols}
If our principal symbol $L_\mathrm{prin}(x,p)$
is sufficiently close to the reference principal symbol
$\mathring L_\mathrm{prin}(x,p)$, then there exists
a unique special unitary matrix-function
$R:M\to\mathrm{SU}(2)$
close to the identity matrix such that
\begin{equation*}
\label{principal symbol via reference principal symbol}
L_\mathrm{prin}(x,p)=R^*(x)\,\mathring L_\mathrm{prin}(x,p)\,R(x)\,.
\end{equation*}
\end{lemma}

The proof of Lemma~\ref{Lemma about two principal symbols}
is given in Section 2 of \cite{action}.

The choice of reference principal symbol $\mathring L_\mathrm{prin}(x,p)$
in our construction is
arbitrary, as long as this principal symbol corresponds to the prescribed metric
$g$, i.e.~as long as we have
$\det\mathring L_\mathrm{prin}(x,p)=-g^{\alpha\beta}(x)\,p_\alpha p_\beta\,$.
It is natural to ask the question: what happens if we choose a different reference
principal symbol $\mathring L_\mathrm{prin}(x,p)$?
The freedom in choosing the reference principal symbol $\mathring L_\mathrm{prin}(x,p)$
is a gauge degree of freedom in our construction and our results
are invariant under changes  of the reference principal symbol,
see Section 6 of~\cite{action} for details.

In order to work effectively with special unitary matrices
we need to choose coordinates on the 3-dimensional Lie group $\mathrm{SU}(2)$.
It is convenient to describe a $2\times2$ special unitary matrix by means of a
spinor $\xi$, i.e.~a pair of complex numbers $\xi^a$, $a=1,2$.
The relationship between
a matrix $R\in\mathrm{SU}(2)$ and a nonzero spinor $\xi$ is given by the formula
\begin{equation}
\label{SU(2) matrix expressed via spinor}
R=
\frac1{\|\xi\|}
\begin{pmatrix}
\xi^1&-\overline{\xi^2}\\
\xi^2&\overline{\xi^1}
\end{pmatrix},
\end{equation}
where the overline stands for complex conjugation
and $\|\xi\|:=\sqrt{|\xi^1|^2+|\xi^2|^2}\,$.

Formula (\ref{SU(2) matrix expressed via spinor})
establishes a one-to-one correspondence between
$\mathrm{SU}(2)$ matrices and nonzero spinors,
modulo a rescaling of the spinor by an arbitrary positive real factor.
The matrices
\begin{equation}
\label{new formula for Pauli matrices in 3D}
\mathring\sigma^\alpha(x):=(\mathring L_\mathrm{prin})_{p_\alpha}(x),\qquad\alpha=1,2,3,
\end{equation}
are the Pauli matrices in our construction.

\begin{remark}
\label{gauge remark}
The gauge-theoretic analysis performed in the current section is
somewhat different from that of Section~\ref{Gauge-theoretic analysis in 4D}.
Namely, the differences are as follows.
\begin{itemize}
\item
In the current section we applied gauge transformations
to the principal symbol whereas in
Section~\ref{Gauge-theoretic analysis in 4D}
we applied gauge transformations to the operator itself.
\item
In the current section we chose a particular principal
symbol as a reference, which led to a somewhat different
definition of Pauli matrices: compare formulae
(\ref{new formula for Pauli matrices in 3D})
and
(\ref{new formula for Pauli matrices}).
\item
In the current section we did not discuss the subprincipal symbol.
\end{itemize}
\end{remark}

\section{A non-geometric representation of the massless Dirac action in 3D}
\label{A non-geometric representation of the massless Dirac action in 3D}

In this section we retain the assumptions
(\ref{m equals two}),
(\ref{n equals three}),
(\ref{principal symbol is trace-free})
and
(\ref{definition of ellipticity}).
In addition, we assume that our manifold $M$ is compact (and without boundary).

We study the eigenvalue problem
\begin{equation}
\label{eigenvalue problem}
Lv=\lambda sv,
\end{equation}
where $s(x)$ is a given positive scalar weight function.
Obviously, the
problem (\ref{eigenvalue problem}) has the same spectrum as the problem
\begin{equation}
\label{eigenvalue problem without weight}
s^{-1/2}Ls^{-1/2}v=\lambda v,
\end{equation}
so it may appear that the weight function $s(x)$ is redundant.
We will, however, work with the eigenvalue problem (\ref{eigenvalue problem})
rather than with (\ref{eigenvalue problem without weight}) because we want our problem
to possess a gauge degree of freedom associated with conformal
scalings of the metric, see Section 5 of \cite{action} for details.
Note also that the subprincipal symbol of the operator
$\,s^{-1/2}Ls^{-1/2}\,$ is $\,s^{-1}L_\mathrm{sub}\,$:
here we are looking at formulae
(\ref{psi transformation of the operator})
and
(\ref{psi transformation of the subprincipal symbol})
with $\psi=-\frac12\ln s$.

The eigenvalue problem (\ref{eigenvalue problem}) can be thought of
as the result of separation of variables
$\,w(x^1,x^2,x^3,x^4)=e^{-i\lambda x^4}v(x^1,x^2,x^3)\,$
in the hyperbolic system
\begin{equation}
\label{dynamic equation most basic with weight}
(-is\,\partial/\partial x^4+L)w=0\,,
\end{equation}
compare with formula (\ref{dynamic equation most basic}).
The operator $-is\,\partial/\partial x^4+L$
appearing in the LHS of  formula (\ref{dynamic equation most basic with weight})
is a special case of the `relativistic' hyperbolic operator
introduced in Section~\ref{Dimension four is special},
its special feature being that it has a naturally defined `time'
coordinate~$x^4$ which does not `mix up' with the `spatial'
coordinates $x^1,x^2,x^3$ (local coordinates on the manifold $M$).

We define the counting function $N(\lambda)$ in the usual way
(\ref{definition of counting function})
as the number of eigenvalues $\lambda_k$
of the problem (\ref{eigenvalue problem})
between zero and a positive~$\lambda$.
The results from Section~\ref{Distribution of eigenvalues and the propagator}
give us explicit formulae for the
coefficients $a$ and $b$ of the asymptotic expansion
(\ref{two-term asymptotic formula for counting function}),
see formulae (\ref{formula for a}) and (\ref{formula for b}),
but these formulae do not have a clear geometric meaning.
Our aim in the current section is to rewrite these formulae
in a geometrically meaningful form.

The coefficients $a$ and $b$ are expressed via
the principal and subprincipal symbols of the operator $L$
as well as the scalar weight function $s(x)$. But in Section
\ref{Gauge-theoretic analysis in 3D} we established
that the principal symbol is described by a metric and
a nonvanishing spinor field $\xi(x)$, with the latter defined modulo
rescaling by an arbitrary positive real function.
We choose to specify the scaling of our spinor field $\xi(x)$
in accordance with the formula
\begin{equation}
\label{normalisation of spinor}
\|\xi(x)\|=s(x).
\end{equation}
The coefficients $a$ and $b$ can now be expressed via
the metric, spinor field
and subprincipal symbol of the operator $L$.

\begin{theorem}
\label{theorem9p1}
The coefficients in the two-term asymptotics
(\ref{two-term asymptotic formula for counting function})
of the counting function
(\ref{definition of counting function})
of the eigenvalue problem
(\ref{eigenvalue problem})
are given by the formulae
\begin{equation}
\label{formula for a in 3D}
a=\frac1{6\pi^2}\int_M
\|\xi\|^3\,\sqrt{\det g_{\alpha\beta}}\ dx\,,
\end{equation}
\begin{equation}
\label{formula for b in 3D}
b=\frac{S(\xi)}{2\pi^2}
-\frac1{4\pi^2}\int_M
\|\xi\|^2\,(\operatorname{tr}L_\mathrm{sub})\,\sqrt{\det g_{\alpha\beta}}\ dx
\,,
\end{equation}
where $S(\xi)$ is the massless Dirac action
with Pauli matrices (\ref{new formula for Pauli matrices in 3D})
and $dx=dx^1dx^2dx^3$.
\end{theorem}

Theorem \ref{theorem9p1} follows from
Theorem 1.1 of \cite{jst_part_b}
and Theorem 1.1 of \cite{action}.
The massless Dirac action is defined in Appendix A of \cite{action}.

Theorem \ref{theorem9p1} allows us to define the concept of
massless Dirac action in dimension three in a non-geometric way.
Namely, if the subprincipal symbol of the operator $L$ is zero,
then the second asymptotic coefficient of the counting function
is, up to the factor $\frac1{2\pi^2}$,
the massless Dirac action.

\section{The covariant subprincipal symbol in 3D}
\label{The covariant subprincipal symbol in 3D}

The assumptions in this section are the same as in
Section~\ref{A non-geometric representation of the massless Dirac action in 3D}.

As pointed out in Remark~\ref{gauge remark},
the gauge-theoretic analysis performed in
Sections
\ref{Gauge-theoretic analysis in 3D}
and
\ref{Gauge-theoretic analysis in 4D}
is somewhat different.
In this section we rewrite
Theorem~\ref{theorem9p1}
in the gauge-theoretic language of Section~\ref{Gauge-theoretic analysis in 4D}.

The central element of the gauge-theoretic analysis
of Section~\ref{Gauge-theoretic analysis in 4D}
was the notion of covariant subprincipal symbol,
see formula (\ref{definition of covariant subprincipal symbol}) for definition.
This definition of covariant subprincipal symbol
works equally well in dimension three, only it becomes
slightly simpler. Namely, observe that if $P$ is a $2\times2$ trace-free
matrix, then $\operatorname{adj}P=-P$.
Hence, the definition of the covariant subprincipal symbol can now
be rewritten as
\begin{equation}
\label{definition of covariant subprincipal symbol 3D}
L_\mathrm{csub}:=
L_\mathrm{sub}
-\frac i{16}\,
g_{\alpha\beta}
\{
L_\mathrm{prin}
,
L_\mathrm{prin}
,
L_\mathrm{prin}
\}_{p_\alpha p_\beta}.
\end{equation}

Let us define the massless Dirac operator on half-densities in
accordance with formula (A.19) from \cite{jst_part_b}.

\begin{lemma}
\label{lemma about massless Dirac operator on half-densities}
Our operator $L$ is a massless Dirac operator on half-densities
if and only if $L_\mathrm{csub}(x)=0$.
\end{lemma}

\begin{proof}
The subprincipal symbol of a massless Dirac operator on half-densities
was calculated explicitly in Section 6 of \cite{jst_part_b}.
Straightforward calculations show that this explicit formula
can be rewritten as
$
\frac i{16}\,
g_{\alpha\beta}
\{
L_\mathrm{prin}
,
L_\mathrm{prin}
,
L_\mathrm{prin}
\}_{p_\alpha p_\beta}
$.
\end{proof}

We can now reformulate Theorem \ref{theorem9p1} in the following
equivalent form

\begin{theorem}
\label{theorem10p2}
The coefficients in the two-term asymptotics
(\ref{two-term asymptotic formula for counting function})
of the counting function
(\ref{definition of counting function})
of the eigenvalue problem
(\ref{eigenvalue problem})
are given by the formulae
\begin{equation}
\label{formula for a in 3D alternative}
a=\frac1{6\pi^2}\int_M
s^3\,\sqrt{\det g_{\alpha\beta}}\ dx\,,
\end{equation}
\begin{equation}
\label{formula for b in 3D alternative}
b=
-\frac1{4\pi^2}\int_M
s^2\,(\operatorname{tr}L_\mathrm{csub})\,\sqrt{\det g_{\alpha\beta}}\ dx
\,.
\end{equation}
\end{theorem}

\begin{proof}
Formula
(\ref{formula for a in 3D alternative})
is a consequence of
formulae
(\ref{formula for a in 3D})
and
(\ref{normalisation of spinor}).
Formula
(\ref{formula for b in 3D alternative})
is a consequence of
formulae
(\ref{formula for b in 3D}),
(\ref{normalisation of spinor}),
(\ref{definition of covariant subprincipal symbol 3D})
and
Lemma~\ref{lemma about massless Dirac operator on half-densities}
from the current paper
and Theorem 1.2 from \cite{jst_part_b}.
\end{proof}

\bibliographystyle{amsplain}

\end{document}